\documentclass[reqno]{amsart}
\usepackage{hyperref} 
\usepackage{graphicx}
\usepackage{tikz}
\usepackage{amsthm,amsmath,verbatim,amssymb}
\usepackage{color}
\usepackage{caption}
\usepackage{arydshln}
\usepackage{tikz}
\theoremstyle{plain}
\newtheorem{lem}{Lemma}

\newtheorem{prop}{Proposition}
\newtheorem{thm}{Theorem}
\newtheorem{cor}{Corollary}

\newcommand{\ily}{i_\lambda(y)}
\newcommand{\ilx}{i_\lambda(x)}

\newcommand{\norm}[1]{\left\Vert#1\right\Vert}

\newcommand{\bc}{\begin{center}}
	\newcommand{\ec}{\end{center}}
\def\real{{\mathbb{R}}}

\numberwithin{equation}{section}
\numberwithin{thm}{section}
\numberwithin{lem}{section}
\numberwithin{cor}{section}
\numberwithin{prop}{section}

\begin{document}
	
	\title[Random waves]{Critical radii and suprema of random waves over Riemannian manifolds}
	
	\author[Feng]{Renjie Feng}
	
	\author[Yao]{Dong Yao}
	\author[Adler]{Robert J.\ Adler}

	\address{Sydney Mathematical Research Institute, The University of Sydney, Australia.}
	\email{renjie.feng@sydney.edu.au}

	
	
	\address{School of Mathematics and Statistics/RIMS, Jiangsu Normal University, China.}
	\email{dongyao@jsnu.edu.cn}
	
	\address{Andrew and Erna Viterbi Faculty of Electrical  Engineering, Technion, Haifa, 32000. }
	\email{radler@technion.ac.il}
	\date{\today}
	\maketitle
	
	\begin{abstract}

		We study random waves on  smooth,  compact, Riemannian manifolds under the {spherical ensemble}.  Our first main result shows that there is a positive universal limit  for the critical radius of a specific deterministic embedding, defined via the eigenfunctions of the  Laplace-Beltrami operator, of such  manifolds into higher dimensional Euclidean spaces. This result enables the application of Weyl's tube formula to derive   the tail probabilities for the suprema of random waves. Consequently, the estimate for the expectation of the Euler characteristic of the excursion set follows directly.		
		
	\end{abstract}

	\section{Introduction}
	
	Random waves on smooth, compact Riemannian manifolds have been a central topic in random geometry, driven by Berry's conjectures  (e.g. \cite{Berry}), which established a connection between random waves and the eigenstates of semi-classical quantum Hamiltonian systems. Recently, they have become a focus of intense studies due to their intrinsic mathematical properties. Although the problems  are often simple to state, rigorous solutions typically demand significant efforts and advanced tools.

	
	A classical problem, with applications in a wide range of (non-mathematical as well as mathematical) disciplines, is to find useful expressions (almost always approximations) for the 
	\emph{excursion probability} 
	\begin{equation}
		\mathbb{P}\left\{\sup _{x \in M} f(x) \geq u\right\}
		\label{excursionprob}
	\end{equation} 
	of  a real-valued random process $f(x)$ defined on a smooth parameter space $M$, typically for high levels $u$.
	
	The setting that interests us is when 
	$({M}, g)$ is a $d$-dimensional, smooth,  compact Riemannian manifold  without boundary. Without loss of generality, we always assume that the metric is normalized so that the associated volume satisfies $$\operatorname V_g(M)=1.$$  Let $\Delta_g$ denote the Laplace-Beltrami operator  with respect to $g$. We consider  an orthonormal basis $\left\{\varphi_n\right\}$ of eigenfunctions of $\Delta_g$, corresponding to eigenvalues $\left\{-\lambda_n^2\right\}$,  ordered and indexed with multiplicity.  That is, 
	$$
	\Delta_g \varphi_n\ =\ -\lambda_n^2 \varphi_n,\quad 0 = \lambda_0 < \lambda_1 \leq \lambda_2 \leq \cdots
	$$
	and 
	$$
	\int_{{M}} \varphi_n(x)  \varphi_m(x) \mathrm{dV}_g(x) \ =\ \delta_{n,m}.
	$$
	
	
	Given $\lambda \in \mathbb{R}_{+}$,  
	we define $\mathcal{E}_{\lambda }$ as the eigenspace 
	\begin{equation}
		\label{eigenspace:equn}
		\mathcal{E}_{\lambda }\ :=\ \left\{\varphi_n: \Delta_g \varphi_n=-\lambda_n^2 \varphi, \, {\lambda_n} \leq \lambda\},\right.
	\end{equation} 
	and we denote its dimension by
	$$k_\lambda:= \operatorname{dim}\mathcal{E}_{\lambda }.$$
	Now, for the spectral projection operator $$K_\lambda:L^2\to \mathcal{E}_{\lambda },$$
	its kernel reads
	\begin{equation}
		\label{kernel:equn}
		K_\lambda(x, y):=\sum_{\lambda_n \leq \lambda} \varphi_n(x) \varphi_n(y).
	\end{equation} 
	Random waves $\Phi_\lambda:M\to \real $   under the \emph{spherical ensemble} are defined as
	\begin{equation} \label{rms}\Phi_\lambda(x)=\sum_{i=1} ^{k_{\lambda }} a_i \varphi_i,\end{equation} 
	where   $a:=(a_1,..., a_{k_\lambda})$ is sampled uniformly on the unit sphere $S^{k_\lambda-1}$  so that
	$$\sum_{i=1}^{k_\lambda}|a_i|^2=1, $$
	and thus, by the orthogonality of the eigenfunctions, the $L^2$-norm of the random waves satisfies
	$$\|\Phi_\lambda(x)\|^2_{L^2}=1.$$ 
	
	Our motivating aim is to estimate the excursion probability \eqref{excursionprob} for $\Phi_\lambda$. Several well-established techniques exist for deriving excursion probabilities of random fields.  The main ones, heavily used for Gaussian random waves, are   the tube method \cite{S1, TK02} and the Euler characteristic method \cite{ TA03, TTA05}.   While  initially looking quite different,  these are known to be basically equivalent in the cases where both methods apply  \cite{TK02}.  In the current paper, we shall  exploit the tube method, for our non-Gaussian, spherical ensemble, setting. 
	
	One demand that all of these techniques have in common is that the underlying manifold $M$ is reasonably smooth, which we assume throughout. 
	A second demand is that the sample paths of the random wave are not too rough.  They  are usually required to be at least $C^2$. This is not a trivial demand for $\Phi_\lambda$, at least as $\lambda\to\infty$, which is the setting that will interest us.
	
	To see why this should be the case, it suffices to consider the simplest one of all examples, $M=S^1$ endowed with the usual Euclidean  metric. In that case the eigenfunctions are sine and cosine functions of ever increasing (with $\lambda$) frequency, and so, while, on the one hand, eigenfunction alone is  infinitely differentiable,  even their first  order derivatives diverge to infinity with $\lambda$.
	
	
	%
	
	
	We conclude the introduction part with some historical remarks and a roadmap to the remainder of the paper.
	
	The connection between the excursion probability  of a  random process and Weyl's tube formula for the excursion set seems to have been first  discovered in a purely statistical setting as early as 1939  \cite{HOT}  and then picked up almost half a century later, again in a mainly statistical setting, in \cite{JJ}. The introduction of this approach to general Gaussian processes is due to 
	\cite{S1}, assuming the existence of a lower bound on the critical radius of certain embedded submanifolds. For the more recent literature, see, for example,  Theorem 10.6.1 in \cite{AT}.  
	
	
	
	In the following section, we will clarify the notions introduced thus far, particularly the concept of the critical radius of a set, and present the main results. Section \ref{cr} focuses on the local Weyl law of the spectral projection kernel. Sections \ref{embeddingproofs:sec} and \ref{proofoftheorem1} address the study of the local geometry and the critical radius of the embedding, respectively. Finally, Section \ref{tail} provides the result and proof concerning the excursion probability.
	
	
	\section{Main results}
	
	Recall that our setting is that of a compact, smooth Riemannian manifold $M$, without boundary.   The functions  $\left\{\varphi_1, \cdots, \varphi_{k_\lambda}\right\}$ form an orthonormal basis for the eigenspace $\mathcal{E}_{\lambda}$ of \eqref{eigenspace:equn}, and  the kernel of  spectral projection operator $K_\lambda:L^2\to \mathcal{E}_{\lambda }$ is 
	given by \eqref{kernel:equn}.
	
	This allows us to define 
	the all-important deterministic mapping which embeds $M$ into a higher dimensional space.

	\subsection{A deterministic embedding}
	Consider the  mapping  $i_\lambda: M \rightarrow \mathbb{R}^{k_\lambda}$ defined by
	\begin{equation}\label{map}
		i_\lambda( x) \ = \   K_\lambda(x,x)^{-1 / 2}\left(\varphi_1(x), \ldots, \varphi_{k_\lambda}(x)\right)^T. 
	\end{equation}  
	It is immediate from \eqref{kernel:equn} that  $\| i_\lambda\|=1$,  so that $i_\lambda$ is actually a mapping into the $(k_\lambda -1)$-dimensional sphere, viz:
	$$i_\lambda: M\to S^{k_\lambda-1}\subset \mathbb{R}^{k_\lambda}.$$
	This mapping has a long history, and it is known that,  for sufficiently large values of $\lambda$, it serves as an embedding of $M$ into $S^{k_\lambda-1}\subset \mathbb{R}^{k_\lambda}$ \cite{P, Ze4, Ze3}, i.e.,
	$$i_\lambda: M\hookrightarrow S^{k_\lambda-1}\hookrightarrow  \mathbb{R}^{k_\lambda}, \quad \lambda\gg 1.$$
	
	While our interest in the map $i_\lambda$ comes from the Riemannian geometry setting, its 
	construction is analogous to that of the classical Kodaira embedding for compact complex manifolds, where eigenfunctions are replaced by holomorphic sections of positive  holomorphic line bundles \cite{GH}.   Similar maps have been  defined via heat kernels \cite{BGG, JMS}, and random versions of it appear in the purely Gaussian literature  \cite{AKTW}.
	
	\subsection{Critical radius}
	The notion of the critical radius of a set has its genesis in the early integral geometry of a century ago, but came into its own under the title  of ``positive reach"  in the  geometric measure theory of \cite{Federer} some fifty years later. The review in Section 2 of \cite{AKTW}  and treatments in
	\cite{AT, TK02} would suffice  for our purposes.
	
	In essence,  the {critical radius} of a set  provides a quantitative  measure of how far, or close, it is to being locally convex, or how  `twisted' it is
	as a subset of its ambient space. 
	
	More precisely, suppose that $M$ is a smooth manifold embedded in an ambient manifold $\mathcal {M}$. The \emph{local critical radius} at a given point $p \in M$ is the maximum distance one can travel along the geodesic in $\mathcal {M}$ that starts at $p$ and is normal to $M$ in $\mathcal {M}$, without encountering a similar vector originating from another point in $M$. The \emph{critical radius} of $M$ is then defined as the infimum of all the local critical radii.  It reflects both local curvature and global topology of the embedding of $M$ in the ambient space $\mathcal  M$. While  the local properties of the embedding are captured through its second fundamental form, the critical radius also takes into account points on $M$ which are far apart in terms of their geodesic distance in $M$ but close in the ambient space $\mathcal {M}$. 
	
	A formal way of computing the  critical radius of a smooth,   compact Riemannian manifold $M$ embedded in the ambient  Euclidean space $\mathbb{R}^k$ is given by
	\begin{equation}\label{critical}
		r(M)=\inf _{x, y \in M} \frac{\|x-y\|^2}{2\left\|\mathcal{P}_y^{\perp}(x-y)\right\|},
	\end{equation} 
	where $\mathcal{P}_y^{\perp}(x-y)$ is the projection of $x-y$ to the normal bundle at $y$ (cf. \cite{JJ, TK02}). 
	
	When $x$ and $y$ in the above infimum are far apart, the ratios inside the infimum in 
	\eqref{critical} are relatively easy to handle. This is the `global' part of the critical radius. However, 
	as $x\to y$, both the numerator and denominator  in formula \eqref{critical} vanish. Nevertheless,  this can be handled by showing  that this limit is related to the second fundamental form and the principal curvature of the embedding. To be more precise, we introduce the inverse function 
	\begin{equation}\label{hxy}h(x,y):=\frac{2\left\|\mathcal{P}_y^{\perp}(x-y)\right\|} {\|x-y\|^2}, \end{equation}
	for which (see A.2. in \cite{KT3}),
	\begin{equation}\label{lim32}\limsup_{x\to y} h(x,y)
		=\max _{w \in T_y(M)^{\perp} \cap S^{k-1}}\left|\lambda_{\max }(y,w)\right|,
	\end{equation}
	where $\left|\lambda_{\max }(y,w)\right|$ is the principal curvature at $y$  that has the largest absolute value with respect to the directions $\pm w$.  
	Consequently, $1 /(\max_\omega \left|\lambda_{\max }(y,w)\right|)$ is an upper bound for the local critical radius at $y $.     
	
	\subsection{Main results} \label{mainrs}

	We now turn our attention to  the local geometry of $i_\lambda(M)$ as a subset of 
	$\mathbb{R}^{k_\lambda}$.  Our first result in this direction is the following
	
	\begin{thm}\label{localrad}    
		Consider the embedding  $i_\lambda(M)\hookrightarrow  \mathbb{R}^{k_\lambda}$. We have,  \begin{equation}\label{them21}\lim_{\lambda\to\infty}\inf_{\operatorname{d_g}(x, y)\leq {(\lambda\log\lambda})^{-1}}
			\frac{\|\ilx-\ily\|^2}{2\left\|\mathcal{P}_y^{\perp}(\ilx-\ily)\right\|}= \sqrt{\frac{d+4}{3(d+2)}}, \end{equation}
		and  the largest absolute value of the principal curvature is asymptotic to $ \sqrt{\frac{3(d+2)}{d+4}}$. 
	\end{thm} 

	While the above result offers hope for a reasonable limit to the critical radius of $i_\lambda (M)$ in $\mathbb R^{k_\lambda}$, a result we shall prove, en passant, later raises problems. 
	
	Specifically, consider the  metric induced on $M$ by the map $i_\lambda$. We will show that the pullback metric satisfies the uniform estimate (see \eqref{dsdsdsds} below) \begin{equation} \label{dlsd}i^*_{\lambda}(g_{E}) =
		\left( \frac{\lambda^2}{d+2}+O(\lambda^{})\right)g\end{equation} 
	as  $\lambda$ becomes sufficiently large, where $g_E$  is the Euclidean metric on   
	$\mathbb{R}^{k_\lambda}$. 
	This  indicates that a geodesic of unit length on $M$ with respect to the Riemannian metric $g$ will be stretched by a factor of order $\lambda$ under the map $i_\lambda$ in the ambient space $\mathbb R^{k_\lambda}$ (or in the ambient space $S^{k_\lambda-1}$ since $\|i_\lambda\|=1$). Therefore, it would appear that the  embedded manifolds $i_\lambda(M)$ become highly twisted
	as $\lambda$ increases, and one might expect that the critical radius of the submanifold $i_\lambda(M)$ in  the ambient spaces   neither converges to a limit nor admits a lower bound. However, the following theorem contradicts this expectation. 
	\begin{thm} \label{main1} Let $r_{\lambda}$  be the critical radius of the embedding $i_\lambda(M)\hookrightarrow \mathbb{R}^{k_\lambda}$. Then $r_{\lambda}$  has the following positive universal limit which  depends only   on the dimension of $M$:  
		\begin{equation} 
			\label{limcri} 
			\begin{split}&\lim_{\lambda\to\infty}r_{ \lambda} \\  =&\inf _{u\geq 0} \frac{ 1-\Gamma(\frac d 2+1)(\frac2{u})^{\frac d2}J_{\frac d2}(u)}{\sqrt{ 2-2\Gamma(\frac d 2+1)(\frac2{u})^{\frac d2}J_{\frac d2}(u) -2\Gamma(\frac d 2+2)\Gamma(\frac d 2+1)\Big[\left(\left(\frac2{u}\right)^{\frac d2}J_{\frac d2}(u)\right)'\Big]^2 }},\end{split}
		\end{equation}  
		where $J_{d /2}$ is the Bessel function  of the first kind of order $ d /2$.
	\end{thm}
	
	Theorem \ref{main1} shows that the critical radius of the submanifold $i_\lambda(M)$ has a positive universal limit in $\mathbb{R}^{k_\lambda}$, which further implies the existence of a positive lower bound for the critical radius of $i_\lambda(M)$ in $S^{k_\lambda-1}$. 
	One possible explanation for this phenomenon is that, as $\lambda$ increases, while the embedded submanifolds are indeed stretched by a factor $\lambda$ in the ambient spaces $S^{k_\lambda-1}$, the growing dimensions of the ambient spaces $S^{k_\lambda-1}$ allow more space for the stretched submanifold, and so self-intersections, and even `near' self-intersections, are avoided. The competition between these two factors eventually leads to the positive limit for the critical radius.


	The proofs of Theorem \ref{localrad} and Theorem \ref{main1} rely on the observation that the critical radius of the embedding $i_\lambda(M)\hookrightarrow\mathbb{R}^{k_\lambda}$  can be expressed in terms of the spectral projection kernel, enabling us to apply the local Weyl law to determine its limit. Remarkably, during the proofs, we will observe that the leading-order term in the local Weyl law determines the limit of the critical radius completely, which turns out to be universal. In other words,  the error term in the local Weyl law, which has the geometry of the Riemannian manifolds $(M,g)$ involved, does not  influence the determination of the limiting critical radius.

	
	Note that  two specific precursors to the current results are \cite{FA} and \cite{FXA}, both treating much the same problem, but only in the special case of $M$ being a unit sphere endowed with the standard Euclidean metric, so that the eigenfunctions are spherical harmonics.
	The complex version of this problem was explored in \cite{S}.

	An analogous result to Theorem \ref{main1}  holds for random embeddings. 
	Consider a smooth and compact Riemannian manifold $M$, and define a centered, unit variance, smooth Gaussian process $f: M \rightarrow \mathbb{R}$. For a given $k \geq 1$, one can  construct a $\mathbb{R}^k$-valued process$$
	f^k(x):=\left(f_1(x), f_2(x), \ldots, f_k(x)\right)
	$$
	made up of the first $k$ processes in an infinite sequence of independent and identically distributed copies of $f$. Then, with probability 1, this process defines a Gaussian random embedding $M\hookrightarrow  \mathbb{R}^k$ for all $k \geq 2 d+1$. To  normalize the map, we define
	$$
	h^k(x):= \frac{f^k(x)}{\|f^k(x)\|}, \quad x \in M.
	$$ 
	This yields  a random embedding of $M\hookrightarrow S^{k-1}$. Let $r_k$ denote the critical radius of $h^k(M)$ in the ambient space $S^{k-1}$. Then the main result in \cite{AKTW} is that, with probability $1$, there exists  a constant $c_f$ depending on the Gaussian process $f(x)$, and well known in the Gaussian literature, such that 
	\begin{equation}
		\label{gausslim}
		\lim_{k\to \infty} r_k =c_f.
	\end{equation} 

	Now we are ready to apply the tube method to derive the excursion probability. Recall that $\|i_\lambda\|=1$. This, together with Theorem  \ref{main1},  implies that there is a lower bound for the critical radius of $i_\lambda(M)$ considered as a submanifold in the ambient space  $S^{k_\lambda-1}$. Let $\rho>0$ denote this new lower bound.  Define the tube of radius $\theta$ around $i_\lambda(M)$ in $S^{k_\lambda-1}$ by
	$$
	\operatorname{Tube}\left(i_\lambda(M), \theta\right) \ := \ \left\{x \in S^{k_\lambda-1}: \min _{y \in i_\lambda(M)}  \Theta(x, y) \leq \theta\right\},
	$$
	where $\Theta(x, y)$ is the angle, or the geodesic distance,  between vectors $x, y \in S^{k_\lambda-1}$. 
	
	Note that, after normalizing, the random waves defined in \eqref{rms} can be expressed in two additional ways. Specifically,
	$$
	\frac{\Phi_\lambda(x)}{\sqrt{K_\lambda(x,x)}} = \left\langle a, i_\lambda(x)\right\rangle= \cos \left(\Theta(a, i_\lambda(x))\right)  \leq  1,
	$$
	where, as before,  the random vector $a=\left(a_1, \ldots, a_{k_\lambda}\right)$ are chosen with respect to the  uniform measure on $S^{k_\lambda-1}$.
	
	Therefore, for $0<\theta<\rho$, we have the formula, 
	\begin{equation}\label{tuebf}\begin{split}\mathbb{P}\left\{\sup_{x\in M}\frac{\Phi_\lambda(x)}{\sqrt{K_\lambda(x,x)}}>\cos\theta\right\}&=\mathbb{P}_{}\left\{\sup _{x\in M}\left\langle a, i_\lambda(x)\right\rangle>\cos \theta\right\} \\&= 
			\frac{\mu_{k_{\lambda} - 1 }(\operatorname{Tube}\left(i_\lambda(M), \theta\right) )}{s_{k_\lambda-1}},\end{split}\end{equation}
	where $ {\mu_{k_\lambda-1}}(T)$ is the  Lebesgue measure of the subset $T\subset  S^{k_\lambda-1}$, and 
	$
	s_{k_\lambda-1} =  {2\pi^{k_\lambda/2}}/{ \Gamma\left({k_\lambda}/{2}\right)}
	$
	is the surface area of the unit sphere $S^{k_\lambda-1}$.
	
	Exploiting Weyl's tube formula for submanifolds of spheres (see \eqref{tubeforsphere} below), 
	we can deduce the following type of large deviation result. 
	\begin{thm}\label{main3} Let $\rho>0$ be the uniform lower bound of the critical radius $i_\lambda(M) \hookrightarrow S^{k_\lambda-1}$.  For any $0<\theta< \rho$, the excursion probability satisfies, 
		$$\begin{aligned}
		&\lim_{\lambda\to\infty}\frac{1}{\lambda^d} \log \mathbb{P}\left\{\sup _{x\in M}\frac{ \Phi_\lambda(x)}{\sqrt{K_\lambda(x,x)}}>\cos\theta \right\}=\frac{\log \sin (\theta)}{(4\pi)^{d/2}\Gamma\left({d}/{2}+1\right)} . \end{aligned}$$
	\end{thm}
	
	Now, for $0<\theta<\rho$, we can study  the Euler characteristic of the following \emph{excursion set},  
	$$A_\lambda(\cos\theta)\ :=\ \left\{x \in M: \frac{\Phi_\lambda(x)}{\sqrt{K_\lambda(x,x)}}>\cos \theta \right\}.$$
	
	Note that, in general,  if $M$ is a smooth,  compact submanifold of   $\mathcal {M}$, for any $p \in  \mathcal {M}$, the intersection between $M$ and a ball of radius $\rho$ around $p$ will be either empty or contractible if $\rho$ is less than the critical radius of $M$ in  $\mathcal {M}$ \cite{S}. 
	This implies that in our case the set 
	$$
	\left\{i_\lambda(x) \in S^{k_\lambda-1}:\left\langle a, i_\lambda(x)\right\rangle  > \cos \theta\right\}
	$$ 
	is either empty or contractible for $0 \leq \theta \leq \rho$. Therefore, for such values of $\theta$,  the expected Euler characteristic of the excursion set  can be expressed in terms of the volume of the tube of the embedding as follows, 
	$$\begin{aligned}
		\mathbb{E} \left\{\chi\left(A_\lambda(\cos \theta)\right)\right\} =&\mathbb{E}\left\{\chi\left\{i_\lambda(x) \in S^{k_\lambda-1}:\left\langle a, i_\lambda(x)\right\rangle>\cos \theta\right\}\right\} \\
		=&\mathbb{P}_{}\left\{\sup _{x\in M}\left\langle a, i_\lambda(x)\right\rangle>\cos \theta\right\}
		\\ =& \frac{\mu_{k_{\lambda} - 1 }(\operatorname{Tube}\left(i_\lambda(M), \theta\right) )}{s_{k_\lambda-1}},
	\end{aligned} $$
	where $\chi(A)$ denotes the Euler characteristic of the set $A$.  Therefore, as a direct consequence of Theorem \ref{main3}, we have,

	\begin{cor}For any $0<\theta< \rho$, 
		the expected Euler characteristic of the excursion set has the limit,  
		$$ \lim_{\lambda\to\infty}	\frac{1}{\lambda^d} \log \mathbb{E} \chi\left\{x\in M:\frac{\Phi_\lambda(x)}{\sqrt{K_\lambda(x,x)}}>\cos\theta\right\} =\frac{ \log \sin (\theta)}{(4\pi)^{d/2}\Gamma\left({d}/{2}+1\right)}. $$
	\end{cor}
	
	In this article, we only consider the random waves under the spherical ensemble. However, another significant class of random waves, extensively studied under the \emph{Gaussian ensemble}, is defined as follows:
	\begin{equation} \label{gausswaves}\Psi_\lambda(x)=\sum_{i=1} ^{k_{\lambda }} b_i \varphi_i,\end{equation} 
	where   $b_i$ are i.i.d. Gaussian random variables with mean 0 and variance $1/k_\lambda$. Consequently, the expected $L^2$-norm satisfies
	$$\mathbb E\|\Psi_\lambda(x)\|^2_{L^2}=1.$$ 
	A natural question is to derive the large deviations for the following excursion probabilities, in the same manner as Theorem \ref{main3},  for fixed $u>0$ as $\lambda\to\infty$, $$\mathbb P\left\{\sup_{x\in M} \frac{\Psi_\lambda(x)}{\sqrt{K_\lambda(x,x)}}>u\right\}.$$
	We aim to address it in future investigations.

	\section{A local Weyl law}
	\label{cr}
	
	The local Weyl law was established by H\"ormander in \cite{H}. It provides  asymptotic expansions for  spectral projection kernels. These expansions  will be crucial for us in the proofs of Theorems  \ref{localrad}  and \ref{main1}, and we collect what we will need in this section.

	Let $\omega_d$ be the volume of the unit ball $B_d\subset \mathbb{R}^d$, 
	\begin{equation}\label{sigma}
		\omega_d:=\frac{\pi^{d / 2}}{\Gamma\left({d}/{2}+1\right)}.
	\end{equation} 
	For $x \in \mathbb{R}^d$  define the function, 
	$$
	\mathcal{B}_d(x):=\frac{1}{\omega_d} \int_{\|\xi\| \leq 1} e^{i\langle x, \xi\rangle} \mathrm{d} \xi.
	$$
	It follows from  the definition of Bessel functions that we also have
	\begin{equation}\label{bfunction}
		\mathcal{B}_d(x)=\frac{1}{\omega_d}\left(\frac{2 \pi}{\|x\|}\right)^{d / 2}  {J}_{\frac{d}{2}}(\|x\|) ,
	\end{equation} 
	where $J_{\frac d 2}(x)$ is the Bessel function  of the first kind with order $d/ 2$,  having the series expansion, 
	\begin{equation}\label{jfunction}J_{\frac{d}{2}}(x)=\sum_{j=0}^{\infty} \frac{(-1)^j}{j!\Gamma(j+\frac{d}{2}+1)} \left(\frac{x}{2}\right)^{2j+\frac{d}{2}}.\end{equation} 
	
	Note that  since $\mathcal B_d$ actually only depends on $x$ via its norm $\|x\|$. Hereafter we shall abuse notation by letting  $\mathcal B_d$ also denote a function on $\real_+$, so that  $\mathcal B_d(x) \equiv \mathcal B_d(\|x\|)$. It should always be clear in what follows to which $\mathcal B_d$ we are referring.
	A fact\footnote{This can be proved by using the relation \eqref{bfunction} and combining two facts: (1)   $\mathcal{B}_d(x)<1$ for all $x>0$; and (2)  $\mathcal{B}_d(x)$ converges to 0 as $\norm{x}\to\infty$.} about Bessel function to be used later is that
	\begin{equation}\label{bels}
		\sup_{x>c}   \Gamma\left(\frac d 2+1\right)
		\left( \frac2{u}\right)^{\frac d2}J_{\frac d2}(u) <1,  \quad \forall \, c>0.
	\end{equation}

	The local Weyl law asserts that there exists some 
	$\eta>0$ such that  if the geodesic distance $ \operatorname{d_g}(x,y)$ between $x$ and $y$ is
	less than $\eta$, then the spectral projection kernel of \eqref{kernel:equn} satisfies the asymptotic expansion (e.g., \cite{BC, G, SHu, X})
	\begin{equation}\label{offdia}
		K_\lambda(x, y)=\frac{\omega_d}{(2 \pi)^d} \lambda^d \mathcal{B}_d(\lambda \operatorname{d_g}(x, y))+O(\lambda^{d-1}).
	\end{equation}  
	Taking $x=y$ in the local Weyl law gives the estimate
	\begin{equation}\label{exp1}
		K_\lambda(x, x)=\frac{\omega_d}{(2 \pi)^d} \lambda^d+O\left(\lambda^{d-1}\right), \end{equation} 
	and thus  the dimension of the eigenspace $\mathcal E_\lambda$ (recall \eqref{eigenspace:equn}) satisfies, 
	\begin{equation}\label{exp2}  \quad k_\lambda=\frac{\omega_d}{(2 \pi)^d} \lambda^d+O\left(\lambda^{d-1}\right).\end{equation}  
	This implies that the $n$-th eigenvalue satisfies  
	\begin{equation}
	\lambda_n \sim 2 \pi\left(\frac{n}{\omega_d}\right)^{1 / d} 
	\end{equation} as $n\to\infty$. 
	
	The asymptotic expansion \eqref{offdia} for the Weyl local law is actually true in the ${C}^{\infty}$-topology, and so,
	if we choose a sufficiently  small geodesic normal coordinate chart,  and multi-indices $\alpha, \beta \in \mathbb{Z}_{\geq 0}^d$,   we have the estimates  \cite{BC, G}
	\begin{equation}\label{exp}
		\partial^{\alpha}_x\partial^\beta_y K_\lambda(x, y)=\frac{\omega_d}{(2 \pi)^d} \lambda^d   \partial^{\alpha}_x\partial^\beta_y   \left[\mathcal{B}_d(\lambda \operatorname{d_g}(x, y))\right]+O\left(\lambda^{d+|\alpha|+|\beta|-1}\right)
	\end{equation}  
	if $ \operatorname{d_g}(x,y)< \eta$. 
	On diagonal, Theorem 1 in \cite{X} gives 
	\begin{equation}\label{offdev1}
		\partial_x^\alpha \partial_y^\beta K_\lambda(x, y)|_{x=y}= \begin{cases}
			C_{d, \alpha, \beta} \lambda^{d+|\alpha|+|\beta|}+{O}\left(\lambda^{d+|\alpha|+|\beta|-1}\right), & \text { if } \alpha- \beta \in 2\mathbb Z^d \\
			{O}\left(\lambda^{d+|\alpha|+|\beta|-1}\right), & \text { otherwise }
		\end{cases},
	\end{equation} 
	where, for   $\alpha, \beta$ such that $ \alpha- \beta \in 2\mathbb Z^d$,  $C_{d, \alpha, \beta}$ is defined as:
	\begin{equation} \label{onder}
		C_{d, \alpha, \beta}=(-1)^{(|\alpha|-|\beta|) / 2} \frac{\prod_{j=1}^d\left(\alpha_j+\beta_j-1\right) ! !}{\pi^{d / 2} 2^{d+\frac{1}{2}|\alpha+\beta| } \Gamma\left(\frac{|\alpha+\beta|+d}{2}+1\right)} .
	\end{equation}  
	As a special case, when $\alpha=\beta$,  we have 
	\begin{equation}
		\sum_{\lambda_n \leqslant \lambda}\left|\partial^\alpha \varphi_n(x)\right|^2=C_{d, \alpha, \alpha} \lambda^{d+2|\alpha|}+{O}\left(\lambda^{d+2|\alpha|-1}\right)
	\end{equation}  uniformly, for $\lambda $ large enough.  
	
	On the other hand,  if the pair $(x,y)$ belongs to some compact set in $M\times M$ disjoint from the diagonal, then,  for $\lambda$ large enough, a uniform upper bound holds, specifically \cite{H}
	\begin{equation}\label{offdev111} K_\lambda(x, y)=O\left(\lambda^{d-1}\right).
	\end{equation}     Additionally, for multi-indices $\alpha, \beta \in \mathbb{Z}_{\geq 0}^d$, we have \cite{X}
	\begin{equation}\label{offdev112}
		\partial^{\alpha}_x\partial^\beta_y K_\lambda(x, y)=O\left(\lambda^{d+|\alpha|+|\beta|-1}\right).
	\end{equation}  
	
	We now have what we need to start the core proofs.
	\section{Local geometry of the embedding}
	\label{embeddingproofs:sec}
	
	In this section, we prove Theorem \ref{localrad} regarding the local geometry at each point of the embedded manifolds. 	\subsection{Estimates for the projection onto tangent space}
	We start  by looking more closely at the geometry of the embedding $i_\lambda(M)$  in the ambient space $\mathbb R^{k_\lambda}$. At the point $y\in M$, by choosing a local coordinate system, the tangent space of the embedding $i_\lambda(M)$ at the image $i_\lambda(y)\in \mathbb R^{k_\lambda}$ is given by 
	$$
	\operatorname{span}_{\mathbb R}\Big\{\frac{\partial \ily}{\partial y_1},\cdots,  \frac{\partial \ily}{\partial y_d}\Big\} \subset T_{i_\lambda(y)}\mathbb R^{k_\lambda}\cong \mathbb R^{k_\lambda}.
	$$
	Given a vector $v \in  \mathbb R^{k_\lambda}$, let $p$ be the  projection of  $v$  onto this tangent space.  We first derive a uniform estimate for the norm of $p$.

	To this end, define the following $k_\lambda\times d$ matrix depending on $\lambda$ and $y$, 
	$$A= \left[ \frac{\partial \ily}{\partial y_1},\cdots,  \frac{\partial \ily}{\partial y_d}\right] .
	$$ 
	Then the projection $p$ of $v$ onto the tangent space is 
	\begin{equation}
		\label{prol}p= A(A^TA)^{-1}A^Tv.
	\end{equation} 
	Defining \begin{equation}\label{psid}\psi_\ell(x): =K_\lambda(x,x)^{-1 / 2} \varphi_\ell(x),\end{equation}  we can express the mapping $i_\lambda$ as
	\begin{equation} 
		\label{iiph}
		i_\lambda(x)=(\psi_1(x),..., \psi_{k_\lambda}(x))^T,
	\end{equation} 
	so that  we have 
	$$A^T=\Big [ \frac{\partial \psi_\ell}{\partial y_i } \Big]_{ 1\leq i \leq d, 1\leq \ell \leq k_\lambda}.$$  
	We further define the following normalized kernel:
	\begin{equation} \label{nke}P_\lambda (x,y):=\langle \ilx, \ily\rangle =\frac{K_\lambda(x,y)}{K_\lambda(x,x)^{\frac 12 }K_\lambda(y,y)^{\frac 12}} .
	\end{equation}
	Now we derive an asymptotic expansion for $\|A(A^TA)^{-1}A^Tv\|^2$, exploiting  the estimates on the spectral projection kernel from the previous section.
	
	First note that 
	$$
	A^TA = \left(  \left\langle \frac{\partial \ily}{\partial y_i},\frac{\partial \ily}{\partial y_j}  \right\rangle \right)_{1\leq i,j\leq d}=  \left(\frac{\partial ^2P_\lambda(x,y)}{\partial x_i \partial y_j}{\ |_{x=y}}\right)_{1\leq i,j\leq d}.
	$$
	By taking the derivatives, we have 
	\begin{equation}\label{twoder}
		\left.  \frac{\partial ^2P_\lambda(x,y)}{\partial x_i \partial y_j} \right|_{x=y}
		=
		\left. \left[\frac{\partial _{y_j}\partial_{x_i}K_\lambda(x,y)}{K_\lambda(y,y)} -\frac{\partial_{y_j}K_\lambda(x,y) \partial_{x_i}K_\lambda(x,y)}{K_\lambda(y,y)^2} \right] \right|_{x=y}. \end{equation}
	Now  choose a geodesic normal coordinate centered at $y$. Then, by the local Weyl law  \eqref{exp1}, \eqref{offdev1} and \eqref{onder}, we can derive the following uniform estimates:
	\begin{equation}
		\label{fird}
		\left. \frac{\partial _{y_j}\partial_{x_i}K_\lambda(x,y)}{K_\lambda(y,y)}  \right|_{x=y} \ =\ \begin{cases} \frac{\lambda^2}{d+2}+ O(\lambda)&\text{ if } i= j
			\\
			O(\lambda)&\text{ if } i\neq j
	\end{cases}\end{equation} 
	and 
	\begin{equation}
		\label{secd} 
		\left.  \frac{\partial_{y_j}K_\lambda(x,y) \partial_{x_i}K_\lambda(x,y)}{K_\lambda(y,y)^2}   
		\right|_{x=y} =O(1)
	\end{equation} 
	as $\lambda \to\infty$. 
	This implies that, 
	for $\lambda$ large enough, we have the estimate 
	\begin{equation}\label{aat}A^TA =  \frac{\lambda^2}{d+2} I_d+O(\lambda),\end{equation} 
	where $I_d$ is the $d\times d$ identity matrix and $O(\lambda)$ denotes a matrix where each entry has a uniform bound of order $O(\lambda)$. This implies the estimate 
	$$(A^TA)^{-1} = {(d+2)} \lambda^{-2} (I_d+O(\lambda^{-1})), $$
	and thus $p$,  the vector of the projection in \eqref{prol}, satisfies the estimate
	\begin{eqnarray}
		\label{normofp}
		\|p\|^2 =&  v^T A[(A^TA)^{-1}]^TA^TA(A^TA)^{-1}A^Tv\\
		=&(1+O(\lambda^{-1})){(d+2)} \lambda^{-2}\| A^Tv\|^2.  \notag
	\end{eqnarray}
	
	\subsection{Proof of Theorem \ref{localrad}} 
	We will derive uniform estimates for the numerator and denominator  in \eqref{critical} separately (with $M$ replaced by $i_{\lambda}(M)$), and denote them by $\mathcal{N}_{\lambda}(x,y)$ and $\mathcal{D}_{\lambda}(x,y)$ (or simply $\mathcal N$ and $\mathcal D$, for brevity) as follows,
	\begin{equation}
		\label{ratio}
		\frac{\|\ilx-\ily\|^2}{2\left\|\mathcal{P}_y^{\perp}(\ilx-\ily)\right\|}=:\frac{\mathcal N_{\lambda}(x,y)}{\mathcal D_{\lambda}(x,y)}=:\frac{\mathcal N}{\mathcal D}.
	\end{equation}   
	For any fixed \(y \in M\), we choose geodesic normal coordinates around \(y\) such that \(y = 0\), and we identify \(x\) in its normal neighborhood with its coordinates \((x_1, \ldots, x_d)\). 
	A simple Taylor expansion around a normal neighborhood of $y=0$  gives 
	\begin{equation}
		\begin{split}
			\label{taylor1}
			i_{\lambda}(x)-i_{\lambda}(y)&= A\begin{bmatrix} x_1\\ \vdots\\ x_d \end{bmatrix}+ \frac 12  \begin{bmatrix} \sum_{i,j}\frac{\partial^2\psi_1(0) }{\partial x_i \partial x_j}x_ix_j  \\ \vdots\\  \sum_{i,j}\frac{\partial^2\psi_{k_\lambda}(0) }{\partial x_i \partial x_j}x_ix_j  \end{bmatrix} +e_\lambda(x),\\
			&=:Ax^T+\frac{1}{2}\mathcal H_{ }(x)+ e_{\lambda}(x), 
		\end{split} 
	\end{equation} 
	where the error term has the integral representation, 
	$$e_\lambda(x)=\left[ x^\alpha \int_0^1 \frac{(1-t)^k}{k!} \sum_{|\alpha|=3} \frac{1}{\alpha!}\partial^\alpha \psi_i(tx) dt\right]_{1\leq i\leq k_\lambda}, $$
	where $\alpha=(\alpha_1,...,\alpha_d)$, $x^{\alpha}=\Pi_{i=1}^d x_i^{\alpha_i}$, $|\alpha|=\alpha_1+\cdots+\alpha_d$ and $\alpha!=\alpha_1\cdots\alpha_d!$. 
	\begin{lem}\label{lem:mainlem}
		The following estimates holds for $\operatorname{d_g}(x,y) \leq (\lambda \log \lambda)^{-1}$,
		\begin{align}
			\norm{A x^T}^2&=\left(  \frac{\lambda^2}{d+2} +O(\lambda)\right)\sum_{i=1}^d x_i^2,\label{axnorm}\\
			\|\mathcal H_{ }(x) \|^2&= \lambda^4  \frac{3\Gamma(\frac d 2+1)}{4\Gamma(\frac d 2+3)}\left(\sum_{i=1}^d x_i^2\right)^2 +O(\lambda^3)\sum_{i=1}^d x_i^4,\label{hnorm}\\\label{113}\|e_\lambda(x)\|&=O(\lambda^ 3) \sum_{i=1}^d |x_i|^3,\\
			\|\mathcal{P}_y \mathcal H_{ }(x) \|^2&= O(\lambda^2)\sum_{i=1}^d x_i^4\label{prods1}.
		\end{align}
	\end{lem}
	\begin{proof}  By \eqref{aat}, the squared norm of $Ax^T$ has the uniform estimate
		$$\begin{bmatrix} x_1& \cdots & x_d \end{bmatrix} A^TA
		{\begin{bmatrix} x_1& \cdots & x_d \end{bmatrix}}^T=\left(  \frac{\lambda^2}{d+2} +O(\lambda)\right)\sum_{i=1}^d x_i^2.$$
		This proves \eqref{axnorm}. 
		For \eqref{hnorm}, we 
		express the squared norm of $\mathcal{H}_{ }(x)$ in terms of the kernel as,
		\begin{align*}
			\|\mathcal H(x) \|^2& =\sum_{\ell=1}^{k_{\lambda}} \left( \sum_{i,j=1}^d\frac{\partial^2\psi_\ell(0) }{\partial x_i \partial x_j}x_ix_j \right)^2\\ &=\sum_{i,j, k,q=1}^d \left(\sum_{\ell=1}^{k_{\lambda}} \frac{\partial^2\psi_\ell(0) }{\partial x_i \partial x_j}\frac{\partial^2\psi_\ell(0) }{\partial x_k \partial x_q}\right)x_ix_jx_kx_q \\&=  \sum_{i,j, k, q =1}^d  \partial^2_{x_i x_j}\partial^2_{ y_k y_q} P_{\lambda}(x,y)|_{x=y=0}x_ix_j x_k x_q. 
		\end{align*} 
		Recall  \eqref{offdev1} and \eqref{onder} regarding the local Weyl law on diagonal, we have 
		\begin{equation} \label{plambda}
			\partial_x^\alpha \partial_y^\beta P_\lambda(x,y) |_{x=y}=\begin{cases}
				\tilde C_{d, \alpha, \beta} \lambda^{|\alpha+\beta|}+{O}\left(\lambda^{|\alpha+\beta|-1}\right), & \text { if } \alpha - \beta \in 2\mathbb Z^d; \\
				{O}\left(\lambda^{|\alpha+\beta|-1}\right), &\text { otherwise, }
			\end{cases}
		\end{equation} 
		where the constant 
		\begin{equation} 
			\tilde C_{d, \alpha, \beta}=(-1)^{(|\alpha|-|\beta|) / 2}\prod_{j=1}^d\left(\alpha_j+\beta_j-1\right) ! ! \frac{ \Gamma\left(\frac{d}2+1\right)}{ 2^{|\alpha+\beta| / 2} \Gamma\left(\frac{|\alpha+\beta|+d}{2}+1\right)}.
		\end{equation}  
		
		We obtain the uniform estimate for $\|\mathcal H_{ }(x) \|^2$ as follows. The fourth order derivative $ \partial^2_{x_i x_j}\partial^2_{ y_k y_l} P_{\lambda}(x,y)|_{x=y}$ is  $\tilde C_{d, \alpha,\beta}\lambda^{4}+O(\lambda^3)$ if $ \alpha-\beta\in 2\mathbb Z^d$; otherwise it is $O(\lambda^3)$. Consider derivatives with respect to the variables $x_1$ and $x_2$. 
		If $ \alpha-\beta\in 2\mathbb Z^d$, there are three different types of choices for $ \alpha$ and $\beta$.  The first choice is $\alpha=(1,1, 0,..., 0)$, $\beta=(1,1, 0,.., 0)$, including $\partial_{x_1}\partial_{x_2}\partial_{y_1}\partial_{y_2}$,  $\partial_{x_1}\partial_{x_2}\partial_{y_2}\partial_{y_1}$,  $\partial_{x_2}\partial_{x_1}\partial_{y_1}\partial_{y_ 2}$, $ \partial_{x_2}\partial_{x_1}\partial_{y_2}\partial_{y_ 1}$;  the second choice is $\alpha=(2,0,0 ,...,0)$ and $\beta=(0,2, 0,...,0)$, including $\partial^2_{x_1}\partial^2_{y_2}$ and $\partial^2_{x_2}\partial^2_{y_1}$.  Each of these choices gives a  term $ x_1^2x_2^2$ with the coefficient 
		$\Gamma(1+d/2)/(4{\Gamma(3+d/2)})$, and the total summation yields, $$\left(\frac 6 4 \frac{\Gamma(\frac d 2+1)}{\Gamma(\frac d 2+3)}\lambda^4 +O(\lambda^3)\right)x_1^2x_2^2.$$
		The last choice is $\alpha=(2,0, 0,.., 0)$ and $\beta=(2,0, 0,.., 0)$, which corresponds to  the term $x_1^4$ with  the coefficient 
		$3{\Gamma(1+d/2)}/4{\Gamma(3+d/2)}$, i.e., 
		$$\left(\frac 3 4 \frac{\Gamma(\frac d 2+1)}{\Gamma(\frac d 2+3)}\lambda^4 +O(\lambda^3)\right)x_1^4.$$
		The above arguments apply to other pairs of derivatives.
		Summing over all derivatives gives the desired estimate \eqref{hnorm}.
		
		The error control \eqref{113} for $\norm{e_{\lambda}(x)}$ can be derived in a similar fashion to \eqref{hnorm}, and we omit the details.
		
		It remains to prove \eqref{prods1}. 	Analogue to   \eqref{prol} and \eqref{normofp},  we have
		\begin{align*} \|\mathcal{P}_y  \mathcal H(x)\|^2=&\|A(A^TA)^{-1}A^T\mathcal H(x) \|^2
			=(1+O(\lambda^{-1})){(d+2)} \lambda^{-2}\| A^T\mathcal H(x)\|^2\\
			=& (1+O(\lambda^{-1})) (d+2) \lambda^{-2}  \left\|\begin{bmatrix} \sum_{i,j=1}^d \partial^2_{ x_i  x_j} \partial_{y_1} P_\lambda(x,y)|_{x=y=0}  x_ix_j  \\ \vdots\\ \sum_{i,j=1}^d \partial^2_{ x_i  x_j} \partial_{y_d} P_\lambda(x,y)|_{x=y=0}  x_ix_j \end{bmatrix}\right\|^2,\end{align*}
		which  proves \eqref{prods1} since 
		$ \partial^2_{ x_i  x_j} \partial_{y_k}P_\lambda(x,y)|_{x=y}=O(\lambda^{ 2})$ by 	\eqref{plambda}. This completes the proof of Lemma \ref{lem:mainlem}. 
		
	\end{proof} 
	
	We now return to the proof of Theorem \ref{localrad}. First consider the numerator $\mathcal N$ in \eqref{ratio}. Using Lemma \ref{lem:mainlem} and the triangle inequality, we have
	$$
	\norm{Ax^T} -\frac{1}{2}\norm{\mathcal{H}(x) }-\norm{e_{\lambda}(x)}\leq 
	\norm{i_{\lambda}(x)-i_{\lambda}(y)}\leq 
	\norm{Ax^T} +\frac{1}{2}\norm{\mathcal{H}(x) }+\norm{e_{\lambda}(x)}.
	$$
	And thus, for $\operatorname{d_g} (x,y) \leq (\lambda \log \lambda)^{-1}$, we get 
	\begin{equation}\label{nuest}
		\mathcal N =\norm{Ax^T}^2=\left(1+O(\log^{-1}\lambda)\right)	\frac{\lambda^2}{d+2} \sum_{i=1}^d x_i^2.
	\end{equation}

	Next, we consider the denominator $\mathcal D$.  Recall the Taylor expansion \eqref{taylor1}. Using the fact
	$Ax^T 
	\in T_{i_\lambda(y)}M$,
	the projection of $i_{\lambda}(x)-i_{\lambda}(y)$   to the normal bundle satisfies 
	\begin{equation*}
		\mathcal{P}_y^{\perp} (i_{\lambda}(x)-i_{\lambda}(y)) =\frac 12  \mathcal{P}_y^{\perp}  \mathcal H(x)+\mathcal{P}_y^{\perp} e_\lambda=\frac{1}{2}\left(\mathcal H(x)-\mathcal{P}_y \mathcal{H}_{\lambda}(x)\right)+\mathcal{P}_y^{\perp}e_{\lambda}(x).
	\end{equation*}
	By Lemma \ref{lem:mainlem}, the triangle inequality and the fact 
	$\norm{\mathcal{P}_y^{\perp}e_{\lambda}(x)}\leq \norm{e_{\lambda}(x)}$, we get
	\begin{equation}\label{111}
		\begin{split}
			\|\mathcal{P}_y^{\perp} (i_{\lambda}(x)-i_{\lambda}(y)) -\frac 12    \mathcal H(x)\| &\leq  \frac{1}{2}\norm{\mathcal{P}_y \mathcal{H}_{}(x)}+
			\|\mathcal{P}_y^{\perp}e_\lambda\| \\
			& =O(\lambda)\sum_{i=1}^d x_i^2 +
			O(\lambda^3)\sum_{i=1}^d x_i^3.
		\end{split}
	\end{equation} 
	Combining this with \eqref{hnorm} for $\norm{\mathcal H_{ }(x)}$, we have that, for $\operatorname{d_g} (x,y) \leq (\lambda \log \lambda)^{-1}$,
	\begin{equation}\label{deest}
		\mathcal D= 2	\|\mathcal{P}_y^{\perp} (i_{\lambda}(x)-i_{\lambda}(y))	\|= \left(1+O\left(\log^{-1}\lambda\right)\right) \sqrt{\frac{3\Gamma(\frac d 2+1)}{4\Gamma(\frac d 2+3)}} \lambda^2
		\sum_{i=1}^d x_i^2. 
	\end{equation}
	Therefore, the estimates  \eqref{nuest}, \eqref{deest} and \eqref{ratio} imply that 
	$$\lim_{\lambda\to\infty} \sup_{\operatorname{d_g}(x, y)\leq (\lambda \log \lambda)^{-1}} \left|\frac{\mathcal N_\lambda(x,y)}{\mathcal D_\lambda(x,y)}- \sqrt{\frac{d+4}{3(d+2)}}\right|=0,$$
	which gives \eqref{them21}.  Furthermore, the above limit, together with \eqref{hxy} and \eqref{lim32} will imply that the largest absolute value of the principal curvature is asymptotic to $ \sqrt{\frac{3(d+2)}{d+4}}$. Now we complete the proof of Theorem \ref{localrad}.

	\section{Critical radius}\label{proofoftheorem1}
	In this section, we prove Theorem \ref{main1}.
	
	\subsection{Kernel representation of the critical radius}
	Treating $i_\lambda(M)$ as a submanifold of $\mathbb R^{k_\lambda}$,  we first express its critical radius in terms of the spectral projection kernel. Recall the definitions of $\mathcal N_{\lambda}(x,y)$ and $\mathcal D_{\lambda}(x,y)$ in \eqref{ratio}, we have:  
	\begin{prop}\label{prop:preradius}
		We have the expression  $$\mathcal N_{\lambda}(x,y)=2(1-P_{\lambda}(x,y)),$$ and the following uniform  estimate over $x,y\in M$,

		\begin{equation}\label{dfor}
			\mathcal D_{\lambda}(x,y)=2\sqrt{ 2-2P_\lambda(x,y) - {(d+2)}\lambda^{-2}\sum_{i=1}^d\Big ( \partial_{y_i} P_\lambda(x,y)   \Big)^2+ O(\lambda^{-1} ) }.		\end{equation} 	\end{prop} 
	\begin{proof}
		From \eqref{map} and  \eqref{nke}, we first have \begin{equation}\label{rc2}
			\begin{split} 
				\mathcal N_{\lambda}(x,y)&=\| i_\lambda(x)-i_\lambda(y)\|^2\\
				&=\langle i_\lambda(x), i_\lambda(x)\rangle-2 \langle i_\lambda(x), i_\lambda(y)\rangle+\langle i_\lambda(y), i_\lambda(y)\rangle\\
				& = 2-2P_\lambda(x,y). 
		\end{split}\end{equation}  
		To compute $\mathcal D_{\lambda}(x,y)$, we first estimate the norm of the projection $\norm{\mathcal{P}_y (i_\lambda(x)-i_\lambda(y))} $.
		By \eqref{prol} and \eqref{normofp}, we get
		\begin{equation}\label{Pybd0}
			\begin{split}
				\norm{\mathcal{P}_y (i_\lambda(x)-i_\lambda(y))}^2=&
				\|A(A^TA)^{-1}A^T(\ilx-\ily)\|^2\\ =&(1+O(\lambda^{-1})){(d+2)} \lambda^{-2}\| A^T(\ilx-\ily)\|^2.
			\end{split}
		\end{equation}
		We claim that
		\begin{equation}\label{Atix}
			A^T(\ilx-\ily)=\Big [\partial_{y_i} P_\lambda(x,y)  \Big]_{1\leq i\leq d}.
		\end{equation}
		Assume \eqref{Atix} for the moment. 
		By  \eqref{exp} and \eqref{offdev112}, we  have the following uniform estimates for all $x,y\in M$, 
		\begin{equation}\label{unifdev}
			\partial_{y_i} P_\lambda(x,y)=O(\lambda),\quad \forall\, i=1,.., d.
		\end{equation}
		Equation \eqref{dfor} now follows by combining \eqref{Pybd0}, \eqref{Atix} and and \eqref{unifdev}.

		It remains to prove \eqref{Atix}. 
		Note that
		\begin{equation}\label{At2}
			A^T(\ilx-\ily)=\Big [ \sum_{\ell =1}^{k_\lambda} \frac{\partial \psi_\ell(y)}{\partial y_i}(\psi_\ell(x)-\psi_\ell(y)) \Big]_{1\leq i\leq d}.
		\end{equation}	
		The first part in the right hand of \eqref{At2} satisfies 
		\begin{equation}\label{p1}
			\sum_{\ell =1}^{k_\lambda} \frac{\partial \psi_\ell(y)}{\partial y_i}\psi_\ell(x)=\partial_{y_i} \left(\sum_{\ell =1}^{k_\lambda} \psi_\ell(y)\psi_\ell(x)\right)=\partial_{y_i} P_\lambda(x,y).
		\end{equation}
		Similarly,  
		\begin{equation}\label{p2}
			\sum_{\ell =1}^{k_\lambda} \frac{\partial \psi_\ell(y)}{\partial y_i}\psi_\ell(y)=\partial_{y_i} P_\lambda(x,y)  {\Big |_{x=y}},
		\end{equation} 
		which must be 0 since $P_{\lambda}(x,y)$ attains its maximum 1 for $x=y$. The claim \eqref{Atix} can be now read off from \eqref{At2}-\eqref{p2}.

	\end{proof} 
	
	\subsection{Proof of Theorem \ref{main1}}
	Now we are ready to prove Theorem \ref{main1}.  	Since the Riemannian manifold $(M, g)$ is smooth, compact, and without boundary, it has a positive injectivity radius \cite{Lee}. This means there exists a uniform constant $\eta > 0$, sufficiently small, such that a geodesic normal coordinate system can always be established around any point on $M$ within a geodesic distance of $\eta$.

	We now divide the analysis into four cases.
	
	{\bf Case	\textcircled{1}}:  $\operatorname{d_g}(x,y)> \eta$. By the uniform  off-diagonal estimates \eqref{offdev111} and \eqref{offdev112},   we have $P_\lambda\to0$ and $\norm{\lambda^{-1}\partial_{y} P_\lambda}\to 0$ as $\lambda\to\infty$ uniformly  over $\operatorname{d_g}(x,y)> \eta$. By Proposition \ref{prop:preradius}, we have
	\begin{equation} 
		\label{firstl}
		\lim_{\lambda\to\infty} \inf_{\operatorname{d_g}(x,y)> \eta} \frac{\mathcal N_{\lambda}(x,y)}{\mathcal D_{\lambda}(x,y)}  =\frac1{\sqrt 2}.
	\end{equation}

	{\bf Case	\textcircled{2}}: $\operatorname{d_g}(x,y)\leq (\lambda\log \lambda)^{-1}$. This case has been studied in Theorem \ref{localrad}, 
	\begin{equation}\label{inf2}
		\lim_{\lambda\to\infty} \inf_{\operatorname{d_g}(x,y)\leq (\lambda\log\lambda)^{-1}}   \frac{\mathcal N_{\lambda}(x,y)}{\mathcal D_{\lambda}(x,y)} =\sqrt{\frac{d+4}{3(d+2)}}.
	\end{equation}
	
	To analyze the behavior $\mathcal N_{\lambda}(x,y)$ and $\mathcal D_{\lambda}(x,y)$ for $(\lambda\log \lambda)^{-1}<\operatorname{d_g}(x,y)\leq \eta$, we note that by the local Weyl law  \eqref{offdia}, \eqref{exp} and the expression \eqref{bfunction}, we have the following uniform estimates for  $\operatorname{d_g}(x,y) \leq \eta$:
	$$P_\lambda(x,y)=\Gamma(\frac d 2+1)\left(\frac2{\lambda \|y\|}\right)^{d/2}J_{\frac d2}(\lambda \operatorname{d_g}(x,y)  )+O(\lambda^{-1}),$$
	and 
	$$\lambda^{-2}\sum_{i=1}^d\Big ( \partial_{y_i} P_\lambda(x,y)   \Big)^2= \Gamma\left(\frac d 2+1\right)^2\Big[\left(\left(\frac2{z}\right)^{d/2}J_{\frac d2}(z)\right)'\Big]^2\Big|_{
		z=\lambda \operatorname{d_g}(x,y)  }+O(\lambda^{-1}).$$
	For convenience, define two functions
	\begin{equation}\label{defd1d2}
		\begin{split}
			\Delta_1(u)=&1-\Gamma\left(\frac d 2+1\right)\left(\frac2{u}\right)^{\frac d2}J_{\frac d2}(u),\\
			\Delta_2(u)=&2-2\Gamma\left(\frac d 2+1\right)\left(\frac2{u}\right)^{\frac d2}J_{\frac d2}(u) \\ &- (d+2) \Gamma\left(\frac d 2+1\right)^2\Big[\left(\left(\frac2{u}\right)^{\frac d2}J_{\frac d2}(u)\right)'\Big]^2.
		\end{split}
	\end{equation}
	Thus, by Proposition \ref{prop:preradius}, for $(\lambda\log \lambda)^{-1}<\operatorname{d_g}(x,y)\leq \eta$, we have
	\begin{equation}
		\mathcal N_{\lambda}(x,y)=2\Delta_1(\lambda \operatorname{d_g}(x,y) ) +O(\lambda^{-1}),
	\end{equation}
	and
	\begin{equation}
		\mathcal D_{\lambda}(x,y)= 2
		\sqrt{\Delta_2(\lambda \operatorname{d_g}(x,y) )+O(\lambda^{-1})}.
	\end{equation}
	Recall the series expansion of $J_{d/2}$ in \eqref{jfunction}, we can pick three small positive constants $c_1, c_2, c_3$ such that
	\begin{equation}\label{dec1}
		\Delta_1(u)\geq c_3u^2 \mbox{ and } \Delta_2(u)\geq c_3 u^4, \,\,\forall\, 0\leq u\leq c_2, 
	\end{equation}
	and (recall \eqref{bels}), 
	\begin{equation}\label{dec2}
		\Delta_1(u)\geq c_1, \,\, \forall\, u\geq c_2. 
	\end{equation} 
	We now divide the regime $(\lambda \log \lambda)^{-1} <
	\operatorname{d_g}(x,y) \leq \eta$ into two cases.

	{\bf Case	\textcircled{3}}: $ (\lambda\log \lambda)^{-1}< \operatorname{d_g}(x,y)\leq c_2\lambda^{-1}$. 
	By \eqref{dec1}, in this case  we have
	$$
	\Delta_1(\lambda \operatorname{d_g}(x,y) ) \geq c_3\log^{-2}\lambda \mbox{ and } 
	\Delta_2(\lambda \operatorname{d_g}(x,y) ) \geq c_3\log^{-4}\lambda. 
	$$
	Thus,
	\begin{equation}\label{nlas}
		\mathcal N_{\lambda}(x,y)=2(1+O(\lambda^{-1}\log^2\lambda ))\Delta_1(\lambda \operatorname{d_g}(x,y) ),
	\end{equation}
	and
	\begin{equation}\label{dlas}
		\mathcal D_{\lambda}(x,y)=2(1+O(\lambda^{-1}\log^4\lambda )) \sqrt{\Delta_2(\lambda \operatorname{d_g}(x,y) )}.
	\end{equation}
	Consequently,
	\begin{equation}\label{inf3}
		\lim_{\lambda\to\infty} \inf_{(\lambda\log\lambda)^{-1}< \operatorname{d_g}(x,y)\leq c_2\lambda^{-1} }   \frac{\mathcal N_{\lambda}(x,y)}{\mathcal D_{\lambda}(x,y)} =\inf_{0< u \leq c_2} \frac{\Delta_1(u)}{\sqrt{\Delta_2(u)}}.
	\end{equation}

	{\bf Case	\textcircled{4}}: $ c_2\lambda^{-1}< \operatorname{d_g}(x,y)\leq \eta$. 
	We partition $[c_2,\infty)$ into two sets
	$$
	\mathcal A_1:=\{u\geq c_2: \Delta_2(u)<c_1^2\}, \quad \mathcal A_2:=\{u\geq c_2: \Delta_2(u)\geq c_1^2\}.
	$$
	(It is possible that either one is empty.)
	If $\lambda \operatorname{d_g}(x,y)\ \in \mathcal A_1$, then by \eqref{dec2},
	for $\lambda$ large, we have
	$$\frac{\mathcal N_{\lambda}(x,y)}{\mathcal D_{\lambda}(x,y)}\geq \frac{1}{\sqrt{2}}.$$
	
	On the other hand, if  $\lambda \operatorname{d_g}(x,y)\ \in \mathcal A_2$, then \eqref{nlas} and \eqref{dlas} will be valid again. Therefore, we get
	\begin{equation}\label{inf4}
		\begin{split} 
			&\lim_{\lambda\to\infty} \inf_{ c_2\lambda^{-1}< \operatorname{d_g}(x,y)\leq \eta } \min\left\{\frac{\mathcal N_{\lambda}(x,y)}{\mathcal D_{\lambda}(x,y)}, \frac{1}{\sqrt{2}}\right\} \\
			=&\lim_{\lambda\to\infty} \inf_{ c_2\lambda^{-1}< \operatorname{d_g}(x,y)\leq \eta,\lambda \operatorname{d_g}(x,y)\in \mathcal A_2} \min\left\{\frac{\mathcal N_{\lambda}(x,y)}{\mathcal D_{\lambda}(x,y)}, \frac{1}{\sqrt{2}}\right\}\\
			=&	\min\left\{   \inf_{ u \in \mathcal A_2} \frac{\Delta_1(u)}{\sqrt{\Delta_2(u)}},\frac{1}{\sqrt{2}} \right\}\\
			=&	   \inf_{ u \geq c_2} \frac{\Delta_1(u)}{\sqrt{\Delta_2(u)}},
		\end{split}
	\end{equation}
	where  we have used the facts that $\Delta_1(u)/\sqrt{\Delta_2(u)}\geq 1/\sqrt{2}$ for $u\in \mathcal A_1$ and that $\Delta_1(u)/\sqrt{\Delta_2(u)} \to 1/\sqrt{2}$ as $u\to\infty$ in the last equality
	
	Theorem \ref{main1} now follows by 
	combining \eqref{firstl}, \eqref{inf2}, \eqref{inf3}, \eqref{inf4} and the definitions of $\Delta_1$ and $\Delta_2$ in \eqref{defd1d2},
	\begin{equation*}
		\begin{split}
			\lim_{\lambda\to\infty}r_{\lambda}&=
			\lim_{\lambda\to\infty} \inf_{x,y\in M} \frac{\mathcal N_{\lambda}(x,y)}{\mathcal D_{\lambda}(x,y)}\\
			&= \lim_{\lambda\to\infty} \min\left\{
			\inf_{\operatorname{d_g}(x,y)> \eta} \frac{\mathcal N_{\lambda}(x,y)}{\mathcal D_{\lambda}(x,y)}, \cdots, 
			\inf_{c_2\lambda^{-1}<\operatorname{d_g}(x,y)\leq  \eta} \frac{\mathcal N_{\lambda}(x,y)}{\mathcal D_{\lambda}(x,y)}
			\right\}\\
			&=  \inf _{u\geq 0} \frac{\Delta_1(u)}{\sqrt{\Delta_2(u)}},
		\end{split}
	\end{equation*}
	where we have used the fact that $\Delta_1(u)/\sqrt{\Delta_2(u)} \to \sqrt{(d+4)/(3(d+2))} $ as $u\to 0$.
	\section{The excursion probability}
	\label{tail}
	
	In this section, we will prove Theorem \ref{main3}. 
	\subsection{Weyl's tube formula and  the induced metric}
	To prove Theorem  \ref{main3} we  will use a version of Weyl's tube formula, which in general  gives a power  series expansion of the volume of a tube. We refer to Chapter 10 of \cite{AT} for a comprehensive study of Weyl's tube formula. 
	
	As a first step, we will need to define the Lipshitz-Killing curvatures of a smooth, compact, $d$-dimensional, Riemannian manifold $(M,g)$. They are given by 
	\begin{equation}\label{generalformula}
		\mathcal{L}_j( M)= \begin{cases}\frac{(-2 \pi)^{-(d-j) / 2}}{\left(\frac{d-j}{2}\right) !} \scalebox{1}{$\displaystyle\int$}_M \operatorname{Tr}\left(R_{  g}^{(d-j) / 2}\right) {\operatorname {dV}}_{ g}, & d-j \text { even, } \\ 0, & d-j \text { odd, }\end{cases}
	\end{equation} 
	where $R_{  g}$ is the curvature tensor of $(M,g)$, and $\operatorname{Tr}(R_g^{(d-j)/2})$ represents the trace of the  power of order $(d-j)/2$ of the curvature tensor. We refer to Section 7.2 of \cite{AT} for the precise definition of the product of the curvature tensor and its trace. In particular, $\mathcal{L}_d(M)$ is the volume $\operatorname{V}_{ g}(M)$ and 
	$\mathcal{L}_0(M)$ is the Euler characteristic, a topological invariant, of $M$.

	The version of Weyl's tube formula that we shall need will be for locally convex submanifolds embedded in spheres.  Let $\widehat  M$ be such a manifold, of dimension $d$, embedded in $S^{N-1}$, endowed with the Riemannian metric  $\widehat  g$  induced from the standard round metric $g_{S^{N-1}}$ on $S^{N-1}$. 
	Then, (e.g.\  Theorem 10.5.7 in \cite{AT}) the volume of a tube around $\widehat  M$ with radius $\theta$ less than its critical radius is given by
	\begin{equation}\label{tubeforsphere}
		{\mu}_{N-1}\left(\operatorname{Tube}( \widehat  M, \theta)\right)=\sum_{j=0}^d F_{N, j}(\theta) \mathcal{L}_j(\widehat  M),
	\end{equation} 
	where
	\begin{equation}\label{constheta}
		F_{N, j}(\theta)=\sum_{k=0}^{\left[\frac{j}{2}\right]}(-4 \pi)^{-k} \frac{1}{k !} \frac{j !}{(j-2 k) !} G_{j-2 k, N-1+2 k-j}(\theta),
	\end{equation}  
	and, for integers $q,b\geq 0$, 
	\begin{equation}\label{Gab}
		G_{q, b}(\theta)=s_{b-1} \int_0^\theta \cos ^q(r) \sin ^{b-1}(r) d r. 
	\end{equation}
	Here $s_{b-1}={2\pi^{b / 2}}/{\Gamma\left({b}/{2}\right)}$ is the surface area of the unit sphere $S^{b-1}\subset \mathbb R^{b}$.
	
	Therefore, to prove Theorem \ref{main3}, we need to derive an asymptotic expansion for the induced metric from the standard round metric on $S^{k_\lambda-1}$ under the map  $i_\lambda: M\hookrightarrow  S^{k_\lambda-1}$. This is the same as the pullback of the Euclidean metric  $g_E$ on $\mathbb R^{k_\lambda}$ under the map $i_\lambda: M\hookrightarrow  \mathbb R^{k_\lambda}$. 	
	For any fixed point $x\in M$, we choose a sufficiently small geodesic normal coordinate around it as before. Recalling \eqref{psid} and \eqref{iiph}, we have
	\begin{equation}\label{pud}\begin{aligned}g_{\lambda}:=&i_\lambda^*(g_E)= \sum_{\ell=1}^{k_\lambda} d\psi_\ell(x)\otimes d\psi_\ell(x) =\sum_{i,j=1}^d \sum_{\ell=1}^{k_\lambda} \frac{\partial\psi_\ell}{\partial x_i}\frac{\partial\psi_\ell}{\partial x_j} dx_i\otimes dx_j 
			\\=
			&\sum_{i,j=1}^d \partial_{x_i}\partial_{y_j} P_\lambda(x,y)|_{x=y}dx_i\otimes dx_j.\end{aligned}\end{equation} 
	Recalling the estimates for the partial derivatives $\partial_{x_i}\partial_{y_j} P_\lambda(x,y)|_{x=y}$ on the diagonal in \eqref{plambda}, we finally obtain
	\begin{equation}\label{dsdsdsds}g_{\lambda}=\frac{\lambda^2}{d+2}\left(\sum_{i=1}^d dx_i\otimes dx_i+O(\lambda^{-1})\right)=\frac{\lambda^2}{d+2}\left(g+O(\lambda^{-1})\right). \end{equation} 
	We note that	the pullback metric $g_{\lambda}$ has also been studied in  \cite{N, P, Ze4, Ze3}, and it holds that
	\begin{equation} \label{cove}(d+2)g_{\lambda}/\lambda^2 \to g \quad \mbox{in $C^0$ topology}. \end{equation} 
	Actually, $g_{\lambda}$ is a stochastic metric. Briefly, given a Gaussian random field $f(x)$ on a smooth compact Riemannian manifold, one can define a Riemannian metric by 
	$$g(X, Y)=\mathbb E (Xf Yf)= X_xY_y C(x,y)|_{x=y}, $$
	where $$C(x,y)=\mathbb E (f(x)f(y))$$  is the covariance kernel of the Gaussian process.  Consequently, it is also obvious that the tools of Riemannian manifolds - connections, curvatures, etc.- can be expressed in terms of the covariance kernel. Remarkably, its curvature tensor is given by (Lemma 12.2.1 in \cite{AT}) 
	\begin{equation}\label{R}-R = \mathbb E\left\{ \left(\nabla^2 f \right)^2\right\}. \end{equation} 
	Here, the square of the Hessian is to be understood in terms of the dot product of tensors (see eq. (7.2.4) in \cite{AT}).   For a detailed study of stochastic metrics, we refer to Section 12.2 of \cite{AT}.  
	
	In our case, by \eqref{pud}, $g_\lambda$ is a stochastic metric defined via the normalized Gaussian random waves (analogue to \eqref{gausswaves}), 
	$$f_\lambda(x)=  K_\lambda(x,x)^{-1 / 2}\sum_{i=1}^{k_\lambda} c_i \varphi_i(x), $$
	where $c_i$ are i.i.d. Gaussian random variables with mean 0 and variance 1.  
	
	Now, the Levi-Civita connection of $g_\lambda$ is given by (see eq.(12.2.6) in \cite{AT}, Section 3.3 in \cite{N})
	$$\Gamma_{ijk} =\partial_{x_i}\partial_{x_j}\partial_{y_k} P_\lambda(x,y)|_{x=y}$$
	and $$\Gamma_{ij}^k=\sum_l g_\lambda^{kl} \partial_{x_i}\partial_{x_j}\partial_{y_l} P_\lambda(x,y)|_{x=y},$$
	where $( g_\lambda^{kl})_{1\leq k, l\leq d}$ is the inverse  $g_\lambda^{-1}$.  The Hessian of the Gaussian process $f_\lambda(x)$ is
	$$H_{ij}(f_\lambda):= \partial_{x_i}\partial_{x_j}f_\lambda- \sum_k \Gamma_{ij}^k\partial_{x_k} f_\lambda.$$
	Now, by formula \eqref{R}, the curvature tensor of the stochastic metric $g_\lambda$ is given by 
	$$R_{ijkl}=-\mathbb E (Q_{ijkl}) $$
	where 
	$$Q_{ijkl}=H_{ij}H_{kl}-H_{kj}H_{il}.  $$
	Again, we fix $x\in M$ and choose normal coordinate around it. By the estimates of local Weyl law \eqref{exp}-\eqref{onder}, as $\lambda$ large enough,   
	the curvature tensor of $g_\lambda$ satisfies the uniform estimate (e.g., eq. (3.17)-(3.19) in \cite{N}) 
	$$|R_{ijkl}|=O(\lambda^3).$$
	Now, let $(e_1,..,e_d)$ be the orthonormal basis with respect to $g_\lambda$,  by \eqref{dsdsdsds}, we have 
	$$(e_1,..,e_d)=\sqrt{d+2}\lambda^{-1}(I_d+O(\lambda^{-1}))\left(\frac{\partial}{\partial x_1}, ..., \frac{\partial}{\partial x_d}\right).$$
	Consequently,  the trace of  the power of  the curvature tensor of $g_\lambda$ has the uniform estimate (cf. Section 7.2 of \cite{AT}), 	
     \begin{equation}\label{bdrg} \operatorname{Tr}(R_{  g_\lambda}^{(d-j) / 2})=O(\lambda^{-(d-j)/2}), \quad \mbox{$d-j$ even}.	\end{equation}

	\subsection{Proof of Theorem \ref{main3}} Now we turn to proving Theorem \ref{main3}. 
	Recall \eqref{tubeforsphere} and \eqref{constheta}, and take  $\widehat M=i_\lambda(M)$, $\widehat g=g_\lambda=i_\lambda^*(g_E)$ and $N=k_\lambda$, i.e., 
	\begin{align}\label{ufor}
		{\mu}_{k_\lambda-1}(\operatorname{Tube}(i_\lambda(M), \theta)) \ &=\sum_{j=0}^d F_{k_\lambda, j}(\theta) \mathcal{L}_j( i_\lambda(M)).\end{align}
	By Laplace's method,  for fixed $q\geq 0$ and $\theta \in (0,\pi/2)$, as $b\to \infty$, we have
	$$
	\int_0^\theta \cos ^q(r) \sin ^{b-1}(r) d r \sim \frac{\cos^q(\theta) \sin^{b-1}(\theta) }{(b-1)\cot (\theta)}. 
	$$ 
	Consequently, 
	the leading contribution to $F_{k_{\lambda},j}(\theta)$ is
	\begin{equation}\label{gfor}	
		G_{j,k_{\lambda}-1-j}(\theta)=
		(1+o(1)) s_{k_{\lambda}-2-j} 
		\frac{\cos^j(\theta) \sin^{k_{\lambda}-2-j}(\theta) }{(k_{\lambda}-2-j)\cot (\theta)}.
	\end{equation}
	For the term $\mathcal{L}_j( i_\lambda(M))$, using \eqref{generalformula},
	\eqref{dsdsdsds} and \eqref{bdrg}, we have 
	\begin{equation}\label{ld0}
		\mathcal{L}_d( i_{\lambda}(M) )=\operatorname{V}_{ g_{\lambda}}(i_{\lambda}(M) )=\left(1+O(\lambda^{-1})\right)\frac{\lambda^d}{(d+2)^{d/2}}\operatorname{V}_{ g}(M),
	\end{equation}
	and 
	\begin{equation}\label{ld1}
		\mathcal{L}_j( i_{\lambda}(M))=O(\lambda^{d-\frac{1}{2}(d-j)} )=O(\lambda^{(d+j)/2 }), \quad \forall \, j\leq d.
	\end{equation}
	By  \eqref{gfor}, \eqref{ld0} and \eqref{ld1},  the right hand side of \eqref{ufor} is dominated by the term $j=d$ so that
	\begin{equation}\label{tubfi}
		\begin{split}
			{\mu}_{k_\lambda-1}(\operatorname{Tube}(i_\lambda(M), \theta)) &\sim F_{k_{\lambda},d}(\theta) 	\mathcal{L}_d( i_{\lambda}(M)\\
			& \sim s_{k_{\lambda}-2-d} \frac{\lambda^d}{k_{\lambda}-2-d }\frac{\sin^{k_{\lambda}-2-d}(\theta) \operatorname{V}_{ g}(M)\cos^d(\theta)  }{(d+2)^{d/2}\cot \theta}.
		\end{split} 
	\end{equation}
	Recalling  \eqref{tuebf},
	\begin{align*}
		&\mathbb{P}\left\{\sup _{x\in M}\frac{ |\Phi_\lambda(x)|}{\sqrt{K_\lambda(x,x)}}>\cos \theta \right\}=   \frac{\mu_{k_{\lambda -1}} \left(\operatorname{Tube}\left(i_\lambda(M), \theta\right) \right)}{s_{k_\lambda-1}}. 
	\end{align*}
	Dividing the second line of \eqref{tubfi} by $s_{k_{\lambda}-1}$ and using the formulas $k_{\lambda}\sim w_d \lambda^d/(2\pi)^d $, $s_{b-1}=2\pi^{b/2}/\Gamma(b/2)$ and $w_d=\pi^{d/2}/\Gamma\left({d}/{2}+1\right)$, we get
	$$
	\lim_{\lambda\to\infty}  \frac{1}{\lambda^d}\log 
	\mathbb{P}\left\{\sup _{x\in M}\frac{ \Phi_\lambda(x)}{\sqrt{K_\lambda(x,x)}}>\cos \theta \right\}= \frac{\log \sin (\theta)}{2^d\pi^{d/2} \Gamma(\frac{d}{2}+1)},
	$$
	as desired. 
	


	\section*{Acknowledgement}      
	We would like to thank Jared Wunsch for many helpful discussions on the local Weyl law. Dong Yao is supported by National Key R\&D Program of China (No. 2023YFA1010101), NSFC grant (No. 12201256) and
	NSF Jiangsu Province grant (No. BK20220677).

\end{document}